\documentclass[preprint,12pt]{elsarticle}

\usepackage{amssymb}
\usepackage{amsthm}

\usepackage{amsmath}
\newtheorem{theorem}{Theorem}
\newtheorem{lemma}{Lemma}
\newtheorem{definition}{Definition}
\newtheorem*{remark}{Remark}
\newtheorem{example}{Example}
\newtheorem{proposition}{Proposition}

 \usepackage{lineno}
\usepackage{hyperref}
\usepackage{tikz-cd}

\journal{Journal of Algebra}

\begin{document}

\begin{frontmatter}



\title{Cyclic homology of Jordan superalgebras and related Lie superalgebras}

\author[label1]{Consuelo Martínez}
\author[label2]{Efim Zelmanov}
\author[label3]{Zezhou Zhang}

\affiliation[label1]{organization={Departamento de Matemáticas, Universidad de Oviedo},
            addressline={C/ Calvo Sotelo s/n}, 
            city={Oviedo},
            postcode={33007},
            country={Spain}}
\affiliation[label2]{organization={SICM, Southern University of Science and Technology},
            city={Shenzhen},
            postcode={518055}, 
            country={China}}
\affiliation[label3]{organization={Department of Mathematics, Beijing Normal University},
            city={Beijing},
            postcode={100875}, 
            country={China}}

\begin{abstract}

We study the relationship between cyclic homology of Jordan superalgebras and second cohomologies of their Tits-Kantor-Koecher Lie superalgebras.

In particular, we focus on Jordan superalgebras that are Kantor doubles of bracket algebras.
The obtained results are applied to computation of second cohomologies and universal central extensions of Hamiltonian and contact type Lie superalgebras over arbitrary rings of coefficients.

\end{abstract}

\begin{keyword}



Jordan algebra \sep superalgebra \sep superconformal algebra

\MSC[2020]  Primary \sep 17B60, 17B66  \sep Secondary \sep 17A70  17C70, 17B68, 81R10
\end{keyword}

\end{frontmatter}


\section*{Introduction}
\label{centralExtLie}

Let $L$ be a (super) algebra that is a Tits-Kantor-Koecher construction of $J$, a Jordan (super) algebra.
   S. Tan (in \cite{15Tan1999TKKA}) and, a bit later, B. Allison, G. Benkart, Y. Gao in \cite{1Allison2000CentralEO} linked second cohomology of the (super) algebra $L$ to cyclic homology of the (super) algebra  $J$, an analog of cyclic homology introduced by A. Connes (see \cite{19Kassel1984KDC}  and \cite{20Kassel1982EC}).
  
For the classical case of superconformal algebras $K(1:n)$, an explicit description of second cohomologies is due to V. G. Kac and J. van de Leur \cite{5Kac1988ClassSCA}.

\medskip

We study, in this paper, the relationship between cyclic homology of Jordan superalgebras and second cohomologies of their Tits-Kantor-Koecher Lie superalgebras. We note that apart from this work and the above mentioned \cite{15Tan1999TKKA} and \cite{1Allison2000CentralEO}, we are not aware of any other references concerning cyclic homology of Jordan (super) algebras.

In particular, we focus on Jordan superalgebras that are Kantor doubles of
bracket algebras. The main results related to cyclic homology of Kantor doubles appear in Section 4.  

The first 3 sections of the paper intend to provide the reader of all needed definitions and preliminary results concerning Lie superalgebras (Section 1), Jordan superalgebras (Section 2) and Brackets (Section 3).

Our interest in Kantor doubles stems from the fact that these Jordan superalgebras are related to superconformal algebras via the Tits-Kantor-Koecher construction.  

The obtained results are applied to the computation of second cohomologies and universal central extensions of Hamiltonian and contact type Lie superalgebras over arbitrary rings of coefficients, in sections 5 through 7. 

 Due to these applications to superconformal algebras -- where the fact of zero characteristic is needed -- we will assume that all considered vector spaces are over a field of zero characteristic, even though all results in Section 2 are valid over fields of characteristics not 2,3.

\section{Preliminaries (Lie superalgebras)}

Let $L=L_{\bar{0}}+L_{\bar{1}}$ be a Lie superalgebra.
A bilinear mapping $L\times L\to F,\; a\times b\mapsto (a|b)\in F$ is called a $2$-cocycle if and only if it is super skew-symmetric and
$$([a, b]|c)+(-1)^{|a|(|b|+|c|)}([b, c]|a)+(-1)^{|c|(|a|+|b|)}([c, a]|b)=0$$
for arbitrary elements $a,b,c\in L$.

\vspace{0.8em}

For an arbitrary bilinear functional $\lambda :L\to F$,
the bilinear mapping $(a|b)=\lambda ([a,b])$ is a $2$-cocycle.
Such cocycles are called $2$-coboundaries.

Let $C^2(L)$ be the vector space of all $2$-cocycles of $L$, and
$B^2(L)$ be the vector space of all $2$-coboundaries of $L$.

\vspace{0.8em}

A Lie superalgebra $L$ is said to be \textit{perfect} if $L=[L,L]$.
Let $L, \widetilde{L}$ be perfect Lie superalgebras.
A surjective homomorphism $\widetilde{L}\stackrel{\varphi}{\longrightarrow}L$
is called a \textit{central extension} if the kernel of $\varphi$ lies in the center
$Z(\widetilde{L})$ of the superalgebra $\widetilde{L}$.

Following ideas of I. Schur \cite{14Schur1904UDG}, H. Garland \cite{3Garland1980TheAT} showed that
for an arbitrary perfect Lie (super)algebra $L$, there exists
a unique universal central extension $\widehat{L}\stackrel{u}{\longrightarrow}L$ such that for an arbitrary central extension  $\widetilde{L}\stackrel{\varphi}{\longrightarrow}L$,
there exists a homomorphism $\chi :\widehat{L}\to\widetilde{L}$ making the diagram

$\begin{tikzcd}
 \widehat{L} \arrow[r,"u"] \arrow[d,"\chi"] & L \\
 \widetilde{L} \arrow[ur,"\varphi"]
\end{tikzcd}$
commutative.

\vspace{0.8em}

The vector space $H^2(L)$ of second cohomologies can be identified with the dual space $Z(\widehat{L})^{*}$, where $Z(\widehat{L})$ is the center of the universal central extension $\widehat{L}$.

The center $Z(\widehat{L})$ can be identified with the space $(L\otimes L)\mathbin{\scalebox{1.25}{\slash}}V$, where $V$ is the subspace spanned by all tensors $$a\otimes b + (-1)^{|a|\cdot|b|}b\otimes a,\; [a,b]\otimes c + (-1)^{|a|(|b|+|c|)}[b, c]\otimes a + (-1)^{|c|(|a|+|b|)}[c, a]\otimes b$$ with $a,b,c\in L$.

\vspace{0.8em}

For further references to universal central extensions see \cite{3Garland1980TheAT, 13Neher2003UCEL}.

\vspace{0.8em}

Suppose that a Lie superalgebra $L$ is $\mathbb{Z}$-graded, $L=\sum _{i\in\mathbb{Z}}L_i = L_{-2}+L_0+L_2$, all homogeneous components $L_i,\; i\neq -2,0,2$ are equal to $0$.
Suppose also that $L_0 = [L_{-2},L_{2}]$.
Let $(x\; |\; y)$ be a $2$-cocycle on $L$.
Then for arbitrary elements $a_{-2},c_{-2}\in L_{-2}; b_2,d_2\in L_2$, we have
\begin{equation}
\label{eq1}
\begin{aligned}
\quad & \quad ([a_{-2},[c_{-2},d_{2}]]\; |\; b_2) + (-1)^{|b|\cdot (|c|+|d|)}(a_{-2}\; |\; [b_2, [c_{-2},d_{2}]]) \\
&= (-1)^{|b|\cdot (|c|+|d|)}([[a_{-2},b_2],c_{-2}]\; |\; d_2) +\\
&\qquad (-1)^{|a|\cdot |c| + |b|\cdot |d|}(c_{-2}\; |\; [[a_{-2},b_2],d_2]).    
\end{aligned}
\end{equation}

In fact, both sides are equal to $(-1)^{|b|\cdot (|c|+|d|)}([a_{-2},b_{2}]\; |\; [c_{-2},d_{2}])$.

\begin{lemma}
    A bilinear mapping $L_{-2}\times L_{2}\to F, a_{-2}\times b_2\mapsto (a_{-2}\; |\; b_2)$ satisfying \eqref{eq1} uniquely extends to a $2$-cocycle on $L$.
\end{lemma}
\begin{proof}
    Without loss of generality, we assume that $L$ is a Lie algebra.
    We define the $2$-cocycle $\psi(x|y)$ on $L$ as follows:
    $$\psi (a_{-2}\; |\; b_2)=(a_{-2}\; |\; b_2),$$
    $$\psi ([a_{-2},b_2]\; |\; [c_{-2},d_2])= ([a_{-2},[c_{-2},d_{2}]]\; |\; b_2) + (-1)^{|b|\cdot (|c|+|d|)}(a_{-2}\; |\; [b_2, [c_{-2},d_{2}]]),$$
    $$\psi (L_i\; |\; L_j)=(0)\text{ for }i+j\neq 0.$$
    The identity \eqref{eq1} implies that
    $$\psi (\sum _i [a_{-2}^{(i)},b_2^{(i)}], \sum _j [c_{-2}^{(j)},d_2^{(j)}])=\sum _{i,j} \psi ([a_{-2}^{(i)},b_2^{(i)}],[c_{-2}^{(j)},d_2^{(j)}]))$$
    is well-defined.
    
\vspace{0.8em}

    Let us check that
    $$\psi ([x,y]\; |\; z) + \psi ([y,z]\; |\; x) + \psi ([z,x]\; |\; y) = 0.$$
    We need to consider only two cases:
    \begin{enumerate}
        \item $x=a_{-2},y=b_2,z=[c_{-2},d_2]$;
        \item $x,y,z\in [L_{-2},L_2]$.
    \end{enumerate}

\vspace{0.8em}

    \textit{Case 1.} We need to show that
    $$\psi ([a_{-2},b_2]\; |\; [c_{-2},d_2]) + \psi ([b_2, [c_{-2},d_2]]\; |\; a_{-2}) + \psi ([[c_{-2},d_2], a_{-2}]\; |\; b_2) = 0.$$

\vspace{0.8em}

    This expression is equal to
    $$\begin{aligned}
\quad & \quad ([a_{-2}, [c_{-2},d_2]]\; |\; b_2) + (a_{-2}\; |\; [b_2, [c_{-2},d_2]]) \\
&- (a_{-2}\; |\; [b_2, [c_{-2},d_2]]) + ([[c_{-2},d_2], a_{-2}]\; |\; b_2) = 0.    
\end{aligned}$$
    
\vspace{0.8em}

    \textit{Case 2.} Let ${\rho}_1,{\rho}_2,{\rho}_3\in [L_{-2},L_2],{\rho}_2=[a_{-2},b_2]$.
    Consider the mapping 
    $$L_{-2}\otimes L_2\stackrel{\widetilde{\psi}}{\longrightarrow}F,\quad x_{-2}\otimes y_2\stackrel{\widetilde{\psi}}{\mapsto}(x_{-2}\; |\; y_2).$$
    The tensor product $L_{-2}\otimes L_2$ is a module over the Lie algebra $[L_{-2},L_2]$.
    
\vspace{0.8em}

    Denote $\widetilde{\rho}_2=a_{-2}\otimes b_2$.
    We need to show that
    $$\psi ([{\rho}_1,{\rho}_2]\; |\; {\rho}_3) + \psi ([{\rho}_2,{\rho}_3]\; |\; {\rho}_1) + \psi ([{\rho}_3,{\rho}_1]\; |\; {\rho}_2) = 0.$$
    Taking into account that
    $$\psi ([{\rho}_1,{\rho}_2]\; |\; {\rho}_3)=-\psi ({\rho}_3\; |\; [{\rho}_1,{\rho}_2]),\; \psi ([{\rho}_2,{\rho}_3]\; |\; {\rho}_1)=\psi ({\rho}_1\; |\; [{\rho}_3,{\rho}_2]),$$
    the last equality follows from
    $$\widetilde{\psi}(-[{\rho}_3, [{\rho}_1,\widetilde{{\rho}_2}]]+[{\rho}_1, [{\rho}_3,\widetilde{{\rho}_2}]]+[[{\rho}_3,{\rho}_1],\widetilde{{\rho}_2}])=0,$$
    which is the Jacobi identity.
    This completes the proof of the lemma.
\end{proof}

\section{Jordan Superalgebras}
\label{JordanSuperalgebras}
\subsection{Preliminaries}
\label{Preliminaries}

\vspace{0.5em}

We start with basic definitions.
A \textit{Jordan algebra} is a vector space $J$ with a binary bilinear equation $(x,y)\mapsto xy$ satisfying the following identities:
$$xy=yx,\quad (x^2y)x=x^2(yx).$$
See \cite{4Jacobson1968StructureJA, 12McCrimmon2004TJA, 18Zhevlakov1882RNA}.

\vspace{0.8em}

Let $V$ be a vector space with countable dimension and let $G=G(V)$ denote the Grassmann (or exterior) algebra over $V$; that is, the quotient of the tensor algebra over the ideal generated by symmetric tensors.
Then $G(V)$ is a $\mathbb{Z}/2\mathbb{Z}$-graded algebra,
$G(V)=G(V)_{\bar{0}}+G(V)_{\bar{1}}$.
Its even part $G(V)_{\bar{0}}$ is the linear span of all tensors of even length, and the odd part $G(V)_{\bar{1}}$ is the linear span of all tensors of odd length.

\vspace{0.8em}

Let ${\cal V}$ be a variety of algebras defined by homogeneous identities (see \cite{4Jacobson1968StructureJA,  18Zhevlakov1882RNA}). A superalgebra $A=A_{\bar{0}}+A_{\bar{1}}$ is called a ${\cal V}$-superalgebra if its Grassmann envelope $G(A)=A_{\bar{0}}\otimes G(V)_{\bar{0}} + A_{\bar{1}}\otimes G(V)_{\bar{1}}$ lies in ${\cal V}$.

\vspace{0.8em}

Thus a Jordan superalgebra is a $\mathbb{Z}/2\mathbb{Z}$-graded algebra $J=J_{\bar{0}}+J_{\bar{1}}$ satisfying the graded identities
$$xy=(-1)^{|x|\cdot |y|}yx\quad\text{(supercommutativity), and}$$
$$\begin{aligned}
\quad & ((xy)z)t + (-1)^{|y|\cdot |z| + |y|\cdot |t| + |z|\cdot |t|}((xt)z)y + (-1)^{|x|\cdot |y| + |x|\cdot |z| + |x|\cdot |t| + |z|\cdot |t|}((yt)z)x \\
&= (xy)(zt) + (-1)^{|y|\cdot |z|}(xz)(yt) + (-1)^{|t|\cdot(|y|+|z|)}(xt)(yz).    
\end{aligned}$$

For an element $a\in J_{\bar{0}}\cup J_{\bar{1}}$, let $R(a)$ denote the operator of right multiplication $R(a):J\ni x\mapsto xa$, where $xR(a)R(b)=(xa)b$. \footnote{In this paper operators are constantly applied to the \textit{right} of vectors.}

\vspace{0.8em}

For arbitrary elements $a,b\in J_{\bar{0}}\cup J_{\bar{1}}$, the operator $D(a,b)=R(a)R(b) - (-1)^{|a|\cdot |b|}R(b)R(a)$ is a derivation of the superalgebra $J$. Such derivations are called inner derivations. The span of all inner derivations is a Lie superalgebra. We will denote it as $\operatorname{Inder}(J)$.

\vspace{0.8em}

For more references on Jordan superalgebras see \cite{10Zelmanov2009JSR}.

\vspace{0.8em}

J. Tits \cite{17Tits1962AA, 16Tits1962CJ} made the following observation. Let $L$ be a Lie (super) algebra, $L_{\bar{0}}$ contains $\mathfrak{sl}(2) = Fe + Ff + Fh, [e,f]=h, [h,e]=2e, [h,f]=-2f$ (we call such triple $e,f,h$ an $\mathfrak{sl}(2)$-triple).
Suppose that the operator $\operatorname{ad}(h): L\to L$ is diagonalizable and has eigenvalues $-2,0,2$, so $L=L_{-2}+L_{0}+L_{2}$ is a direct sum of eigenspaces. Then $J=(L_2, a\cdot b=[[a,f],b])$ is a Jordan (super)algebra.

Moreover, J. Tits \cite{17Tits1962AA, 16Tits1962CJ}, I. Kantor \cite{6Kantor1972GeneralizeJA} and M. Koecher \cite{9Koecher1967ImbeddingJL} showed that every Jordan (super)algebra can be obtained in this way.
The corresponding Lie (super)algebra $L$ is not unique. We will recall the constructions of two Lie superalgebras with these properties: the largest (universal) one and the smallest (reduced) one.

\vspace{0.8em}

For elements $x,y,z\in J_{\bar{0}}\cup J_{\bar{1}}$ of a Jordan superalgebra $J$, we consider their Jordan triple product $$\{ x,y,z \} = (xy)z + x(yz) - (-1)^{|x|\cdot |y|}y(xz).$$
Fix elements $y,z$ and consider the operator $V(y,z): x\mapsto \{ x,y,z \}$. Then $V(y,z)=D(y,z)+R(yz)$.

\subsection{The universal Tits-Kantor-Koecher construction}
\label{universalTKK}

\vspace{0.5em}

\begin{remark}
      We will introduce in the following, the TKK construction of a unital Jordan (super)algebra in the shortest way, using bases: even though it is possible to do it in  basis-free ways.
\end{remark}

Let $J$ be a Jordan (super)algebra with an identity $1$. Consider a basis $\{ e_i \} _{i\in I}$ of $J$. Let $$\{ e_i,e_j,e_k \} = \sum _{t} \gamma _{ijk}^te_t,\; \gamma _{ijk}^t\in F.$$
Define a Lie (super)algebra $\widehat{L}$ by generators $\{ x_i^{-}, x_j^{+} \} _{i,j}$ and relations
$$[[x_i^{+}, x_j^{-}],x_k^{+}]=2 \sum _{t} \gamma _{ijk}^t x_t^{+},$$
$$[[x_i^{-}, x_j^{+}],x_k^{-}]=2 \sum _{t} \gamma _{ijk}^t x_t^{-},$$
$$[x_i^{-}, x_j^{-}] = [x_i^{+}, x_j^{+}] = 0.$$

The resulting algebra $\widehat{L}$ is $\mathbb{Z}$-graded (let $\deg x_i^{+}=2, \deg x_i^{-}=-2$).
Moreover, $\widehat{L}$ is spanned by $x_i^{+}, x_j^{-}, [x_i^{+}, x_j^{-}]$, which implies that $\widehat{L_i}=(0)$ for $i\neq -2,0,2$.

\vspace{0.8em}

Choose $x_1=1$. Then $x_1^{+}, x_1^{-}, [x_1^{+}, x_1^{-}]$ is an $\mathfrak{sl}(2)$-triple, $J^{+}=\operatorname{span}(x_i^{+}, i\in I)=\widehat{L_2},\; J^{-}=\operatorname{span}(x_i^{-}, i\in I)=\widehat{L_{-2}}$ are eigenspaces of $\widehat{L}$ with respect to $\operatorname{ad}(h)$.

\vspace{0.8em}

The (super)algebra $\widehat{L}=TKK(J)$ is universal in the following sense. Let $L'$ be a Lie (super)algebra, $L'_{\bar{0}}\supset \mathfrak{sl}(2)= Fe' + Ff' + Fh',\; L'=L'_{-2}+L'_{0}+L'_{2},\; L'_0=[L'_{-2}, L'_{2}],\; (L^{-}_{2},\circ )\cong J$.
Then there exists an epimorphism $$\varphi :TKK(J)\stackrel{\varphi}{\longrightarrow}L',\quad \varphi (x_1^{+})=e',\; \varphi (x_1^{-})=f',$$
the $\operatorname{ker}\varphi$ lies in the center of $TKK(J)$.

\vspace{0.8em}

It is easy to see that a Lie (super)algebra with this universal property is unique. In particular, the construction above does not depend on a choice of a basis in $J$.

\subsection{The reduced Tits-Kantor-Koecher construction}
\label{reducedTKK}

\vspace{0.5em}

Again, let $J$ be a Jordan (super)algebra with $1$.
Consider two copies of the vector space $J$: $J^{-}, J^{+}$, and their direct sum $J^{-}\oplus J^{+}$.
For arbitrary elements $a^{-}\in J^{-},b^{+}\in J^{+}$, consider the linear operator $$\delta (a^{-},b^{+}):J^{-}\oplus J^{+} \to J^{-}\oplus J^{+},\quad x^{-}\mapsto -(-1)^{|a||b|}\{ x,b,a \} ^{-},\; x^{+}\mapsto \{ x,a,b \} ^{+}.$$

\vspace{0.8em}

It follows from Jordan identities that the span $\delta (J^{-},J^{+})$ is a Lie (super)algebra.
The direct sum of vector spaces $L=\overline{TKK(J)}=J^{-}\oplus\delta (J^{-},J^{+})\oplus J^{+}$ with the operations $[J^{-},J^{-}]=[J^{+},J^{+}]=(0),\; [a^{-},b^{+}]=2\delta (a^{-},b^{+})$ is a Lie superalgebra.
The elements $1^{-},1^{+},2\delta (1^{-},1^{+})$ form an $\mathfrak{sl}(2)$-triple. The Lie superalgebra $\overline{TKK(J)}$ is called the reduced Tits-Kantor-Koecher Lie (super)algebra of $J$. It is easy to see that $Z(\overline{TKK(J)})=(0)$.

\subsection{Cyclic homology of Jordan (super)algebras}
\label{cyclicHomology}

\vspace{0.5em}

Let $J$ be a Jordan (super)algebra with the identity element $e$.
Let $L=\overline{TKK(J)}=J^{+}+\delta (J^{+},J^{-})+ J^{-}$.
Let $(x|y)$ be a $2$-cocycle on the (super)algebra $L$.
The mapping $J^{+}\otimes e^{-}\to [J^{+}, e^{-}],\; a^{+}\otimes e^{-} \mapsto [a^{+}, e^{-}]$ is a bijection.
Indeed, if $[a^{+}, e^{-}]=0$, then $[[a^{+}, e^{-}],e^{+}]=2a^{+}=0$.
Hence we can define a linear functional $\lambda : L\to F$ such that for an arbitrary element $a\in J$ we have $(a^{+}\; |\; e^{-})=\lambda ([a^{+},e^{-}])$.
Subtracting the coboundary corresponding to $\lambda$ we can assume that
\begin{equation}
\label{eq2}
    (J^{+}\; |\; e^{-})=0.
\end{equation}

\vspace{0.8em}

Let $C_0^2(L)$ be the vector space of $2$-cocycles satisfying \eqref{eq2}.
We have shown that $C^2(L)=C_0^2(L) + B^2(L)$.

\vspace{0.8em}

\begin{lemma}
    Let $(x|y)$ be a $2$-cocycle from $C_0^2(L)$. Then $(J^{-}\; |\; e^{+})=0$.
\end{lemma}
\begin{proof}
    An arbitrary element from $J^{-}$ can be represented as $[[a^{+}, e^{-}],e^{-}],~ a^{+}\in J^{+}$.
    Applying the identity $([x,y]\; |\; z) = (x\; |\; [y,z])- (-1)^{|x|\cdot |y|}(y\; |\; [x,z])$ twice, we get
    $$\begin{aligned}
        ([[a^{+}, e^{-}],e^{-}]\; |\; e^{+}) &= ([a^{+}, e^{-}]\; |\; [e^{-},e^{+}]) - (e^{-}\; |\; [[a^{+}, e^{-}],e^{+}]) \\
&= (a^{+}\; |\; [e^{-},[ e^{-},e^{+}]]) - (e^{-}\; |\; [a^{+}, [e^{-},e^{+}]]) - (e^{-}\; |\; [[a^{+}, e^{-}],e^{+}]) \\
&= -2(a^{+}\; |\; e^{-}) - 4(e^{-}\; |\; a^{+}) = 2(a^{+}\; |\; e^{-}) = 0.
    \end{aligned}$$
\end{proof}

From now on we consider a $2$-cocycle $(x\; |\; y)$ such that
\begin{equation}
\label{eq3}
    (J^{-}\; |\; e^{+})=(J^{+}\; |\; e^{-})=(0).
\end{equation}

Define the bilinear mapping $J\times J\to F$ as $(a\; |\; b)=(a^{+}\; |\; b^{-})$.

\begin{lemma}
    For arbitrary elements $a,b\in J_{\bar{0}}\cup J_{\bar{1}}$, we have $(a\; |\; b) + (-1)^{|a|\cdot |b|} (b\; |\; a)=0$.
\end{lemma}
\begin{proof}
    $$\begin{aligned}
    (a^{+}\; |\; b^{-}) &= -\frac{1}{2} ([[a^{-},e^{+}],e^{+}]\; |\; b^{-})\\
&= -\frac{1}{2} ( ([a^{-},e^{+}]\; |\; [e^{+},b^{-}]) - (e^{+}\; |\; [[a^{-},e^{+}],b^{-}]) ) \\
&= -\frac{1}{2} ([a^{-},e^{+}]\; |\; [e^{+},b^{-}]) \\
&= -\frac{1}{2} ( (a^{-}\; |\; [e^{+},[e^{+},b^{-}]]) - (e^{+}\; |\; [a^{-},[e^{+},b^{-}]]) ) \\
&= (a^{-}\; |\; b^{+}),
    \end{aligned}$$
    since $[e^{+},[e^{+},b^{-}]]=-2b^{+}$. Now, $(a^{-}\; |\; b^{+})=-(-1)^{|a|\cdot |b|}(b^{+}\; |\; a^{-})$, which completes the proof of the lemma.
\end{proof}

\begin{lemma}
    For arbitrary elements $a,b,c\in J$, we have
    \begin{equation}
    \label{eq4}
        (ab\; |\; c) + (-1)^{|a|(|b|+|c|)}(bc\; |\; a) + (-1)^{|b|\cdot |c|}(ac\; |\; b) = 0.
    \end{equation}
\end{lemma}
\begin{proof}
    Let $a,b,c,d\in J_{\bar{0}}\cup J_{\bar{1}}$. The equality \eqref{eq1} applied to $\left([a^{+},b^{-}] ~|~ [c^{+},d^{-}]\right)$ yields:
     $$\begin{aligned}
    \quad & ([[a^{+},b^{-}],c^{+}]\; |\; d^{-}) - (-1)^{|c|\cdot |b| + |c|\cdot |a| + |a|\cdot |b|} (c^{+}\; |\; [[b^{-},a^{+}],d^{-}])\\
&= - (-1)^{|b|\cdot |c| + |b|\cdot |d| + |c|\cdot |d|} ([a^{+},[d^{-},c^{+}]]\; |\; b^{-}) + (a^{+}\; |\; [b^{+},[c^{+},d^{-}]]).
    \end{aligned}$$
    In the language of Jordan triple products, it looks as:
     $$\begin{aligned}
    \quad & (\{ a,b,c \}\; |\; d) - (-1)^{|c|\cdot |b| + |c|\cdot |a| + |a|\cdot |b|} (c\; |\; \{ b,a,d \} )\\
&= - (-1)^{|b|\cdot |c| + |b|\cdot |d| + |c|\cdot |d|} (\{ a,d,c \}\; |\; b) + (a\; |\; \{ b,c,d \}).
    \end{aligned}$$
    Equivalently,
    \begin{equation}
    \label{eq5}
        \begin{aligned}
    \quad & \quad (\{ a,b,c \}\; |\; d) + (-1)^{|a|\cdot |b| + |d|\cdot |c|} (\{ b,a,d \} \; |\; c)\\
&+ (-1)^{|b|\cdot |c| + |b|\cdot |d| + |c|\cdot |d|} (\{ a,d,c \}\; |\; b)\\ 
&+ (-1)^{|a|(|b|+|c|+|d|)}(\{ b,c,d \} \; |\; a) = 0.
    \end{aligned}
    \end{equation}
    Let $d=e$. Then we get the assertion of the lemma.
\end{proof}

\begin{definition}
Let $J$ be a Jordan (super) algebra. We call a (super) skew-symmetric bilinear mapping $J\times J\to F$ a cyclic cocycle if the identity \eqref{eq4} holds.
\end{definition}

\begin{remark}
    The identity \eqref{eq4} immediately implies that $(J\; |\; e)=(0)$.
\end{remark}

Let $C(J)$ be the vector space of all cyclic cocycles on $J$.
We have described the linear mapping $C_0^{2}(L)\stackrel{\mu}{\longrightarrow}C(J)$.
Since a $2$-cocycle on $L$ is uniquely determined by its values on $J^{+}\otimes J^{-}$, it follows that the mapping $\mu$ is injective.

\begin{lemma}
    The mapping $\mu$ is bijective.
\end{lemma}
\begin{proof}
    We need to show that \eqref{eq4} implies \eqref{eq1}.
    Again without loss of generality, we will consider the case of algebras, not superalgebras.
    Choose arbitrary elements $a,b,c,d\in J$.
    The left hand side of \eqref{eq1} is:
    $$\begin{aligned}
    \quad & \quad ((ab)c^{(1)} + a(bc)^{(2)} - b(ac)^{(3)} \; |\; d) + ((ab)d^{(1)} + b(ad)^{(4)} - (bd)a^{(5)} \; |\; c) \\
&+ ((ad)c^{(4)} + a(dc)^{(6)} - (ac)d^{(3)} \; |\; b) + ((bc)d^{(2)} + b(cd)^{(6)} - (bd)c^{(5)} \; |\; a).
    \end{aligned}$$
    The upper script is the number of the ``grouping".
    For example, in the group $(1)$, the identity \eqref{eq2} is applied to three elements $ab,c,d$.
    We get $$\begin{aligned}
    \quad & \quad - (dc\; |\; ab)^{(1)} - (ad\; |\; bc)^{(2)} + (bd\; |\; ac)^{(3)} \\
&- (bc\; |\; ad)^{(4)} + (ac\; |\; bd)^{(5)} - (ab\; |\; dc)^{(6)}.
    \end{aligned}$$
    This expression is equal to $0$ since the cocycle $(x\; |\; y)$ is skew-symmetric.
\end{proof}

Recall that $\operatorname{Inder}(J)$ is the span of all inner derivations of a Jordan superalgebra $J$, and $\operatorname{Inder}(J)$ is a Lie superalgebra.
For arbitrary elements $a,b,c\in J_{\bar{0}}\cup J_{\bar{1}}$, we have
\begin{equation}
\label{eq6}
    D(ab,c) + (-1)^{|a|(|b|+|c|)}D(bc,a) + (-1)^{|b|\cdot |c|}D(ac,b) = 0
\end{equation}
(see \cite{2Cantarini2007Classification}).

\vspace{0.8em}

Let $\lambda :\operatorname{Inder}(J)\to F$ be a linear functional.
From \eqref{eq6} it follows that $(a\; |\; b)=\lambda (D(a,b))$ is a cyclic cocycle of the Jordan super-algebra $J$.
We call such cocycles \textit{cyclic coboundaries}.

\vspace{0.8em}

Let $B(J)$ be the vector space of all cyclic coboundaries of $J$.
Following S. Tan \cite{15Tan1999TKKA} and B. Allison, G. Benkart, Y. Gao \cite{1Allison2000CentralEO}\footnote{In \cite{1Allison2000CentralEO}, $HC(J)$ is called the full skew-dihedral homology group.}, we call $HC(J)=C(J)/B(J)$ the \textit{cyclic homology} space of the Jordan (super) algebra $J$.

\vspace{0.8em}

In the important case of $J$ coming from associative algebra, $HC(J)$ is the first cyclic homology group of A. Connes.

\begin{lemma}
    $\mu (B^2(L)\cap C_0^2(L)) = B(J)$.
\end{lemma}
\begin{proof}
    Let a $2$-cocycle $\varphi\in C_0^2(L)$ be a coboundary. It means that there exists a linear functional $\lambda :L\to F$ such that $\varphi (a\; |\; b)=\lambda ([a,b])$. In particular, $\lambda ([J^{+},e^{-}])=\lambda ([J^{-},e^{+}])=(0)$.

\vspace{0.8em}

    Let $\psi = \mu (\varphi)$. Then for arbitrary elements $a,b\in J$, we have $\psi (a\; |\; b)=\lambda ([a^{+},b^{-}])$.
    By Lemma 3, $\lambda ([a^{+},b^{-}]) = - (-1)^{|a|\cdot |b|}\lambda ([b^{+},a^{-}])=\frac{1}{2}(\lambda([a^{+},b^{-}])-(-1)^{|a|\cdot |b|}\lambda ([b^{+},a^{-}]))$.

\vspace{0.8em}

    Define the linear functional $\delta : \operatorname{Inder}(J)\to F,\quad \delta (D(a,b))=\lambda ([a^{+},b^{-}])$.
    
\vspace{0.8em}

    We need to show that $\sum _i D(a_i,b_i)=0$ implies $\lambda (\sum _i [a_i^{+},b_i^{-}])=0$. Without loss of generality, we assume that $|a_i|+|b_i|$ does not depend on $i$. We will show that $\sum _i D(a_i,b_i)=0$ implies $\sum _i ([a_i^{+},b_i^{-}] - (-1)^{|a_i|\cdot |b_i|}[b_i^{+},a_i^{-}]) \in Z(L)$. Therefore $\sum _i ([a_i^{+},b_i^{-}] - (-1)^{|a_i|\cdot |b_i|}[b_i^{+},a_i^{-}]) = 0$ as $L$ is the reduced Tits-Kantor-Koecher Lie superalgebra of $J$.

\vspace{0.8em}

    For an arbitrary element $c\in J$, we have
    $$\begin{aligned}
    \quad & \quad [\sum _i ([a_i^{+},b_i^{-}] - (-1)^{|a_i|\cdot |b_i|}[b_i^{+},a_i^{-}]),c^{+}] \\
&=2 \sum _i (\{ a_i,b_i,c \} ^{+} - (-1)^{|a_i|\cdot |b_i|} \{ b_i,a_i,c \} ^{+} ) \\
&= -4(-1)^{|c|(|a_i|+|b_i|)}(c\sum _i D(a_i,b_i))^{+} = 0.
    \end{aligned}$$

    Similarly,
    $$[\sum _i ([a_i^{+},b_i^{-}] - (-1)^{|a_i|\cdot |b_i|}[b_i^{+},a_i^{-}])),J^{-}]=(0).$$

    We showed that
    $$\mu (B^2(L)\cap C_0^2(L)) \subseteq B(J).$$

    Now let $\delta : \operatorname{Inder}(J)\to F$ be a linear functional. We will show that for arbitrary elements $a_i,b_i\in J$, $\sum _i [a_i^{+},b_i^{-}]=0$ implies $\sum _i D(a_i,b_i)=0$.

\vspace{0.8em}

    For an arbitrary element $x\in J$, we have $\{ x,a,b \} = (xa)b + x(ab) - (-1)^{|a|\cdot |b|}(xb)a = x(D(a,b)+R(ab))$. Hence $$[x^{-},\sum _i ([a_i^{+},b_i^{-}]]=2(x(\sum _i D(a_i,b_i) + R(\sum _i a_ib_i)))^{-},$$ which implies $$\sum _i D(a_i,b_i) + R(\sum _i a_ib_i) = 0.$$

    Applying the left hand side operator to $x=e$, we will get $\sum _i a_ib_i = 0$, and therefore, $\sum _i D(a_i,b_i) = 0$.

\vspace{0.8em}

    Define the linear functional $\lambda : L\to F$ as follows: $$\lambda (J^{+}) = \lambda (J^{-}) = (0),\quad \lambda ([a_i^{+},b_i^{-}])=\delta (D(a_i^{+},b_i^{-})).$$
    The cocycle corresponding to $\lambda$ lies in $C_0^2(L)$ and $\mu (\lambda )=\delta$.
\end{proof}

We proved the following theorem.
\begin{theorem}
    $H^2(L)\cong HC(J)$.
\end{theorem}
\begin{example}
    Let $J$ be a finite dimensional simple Jordan algebra with the trace $\operatorname{tr}: J\to F$ (see \cite{4Jacobson1968StructureJA}). Let $F[t,t^{-1}]$ be the algebra of Laurent polynomials, let $\operatorname{Res}f(t)$ denote the coefficient of the polynomial $f(t)$ at $t^{-1}$.
    Then,
		$$(J\otimes F[t,t^{-1}])\times (J\otimes F[t,t^{-1}]) \to F,\; (af(t)\; |\; bg(t))=\operatorname{tr}(ab)\operatorname{Res}(f'(t)g(t)),$$ 
		$a,b\in J$, is a cyclic cocycle on the Jordan algebra $J\otimes F[t,t^{-1}]$ and 
		$$\dim _F HC(J\otimes F[t,t^{-1}]) = 1.$$
\end{example}

\section{Brackets}
\label{brackets}

Let $A=A_{\bar{0}}+ A_{\bar{1}}$ be an associative commutative superalgebra. A binary bilinear operation $[\; ,\; ]:A\times A\to A$ is called a \textit{Poisson bracket} if
\begin{enumerate}
    \item $(A,[\; ,\; ])$ is a Lie superalgebra,
    \item $[ab,c]=a[b,c]+(-1)^{|b|\cdot |c|}[a,c]b$.
\end{enumerate}
\begin{example}
    Let $A=F[p_1,\dots ,p_n,q_1,\dots ,q_n]$ be a polynomial algebra in $2n$ variables. Then $$[f,g]=\sum _{i=1}^{n}(\frac{\partial f}{\partial p_i}\frac{\partial g}{\partial q_i} -\frac{\partial f}{\partial q_i} \frac{\partial g}{\partial p_i} )$$ is a Poisson bracket in $A$.
\end{example}
\begin{example}
    Let $G(n)$ be the Grassmann algebra on an $n$-dimensional vector space $V$. Let $\xi _1,\dots ,\xi _n$ be a basis of $V$. The bracket $[\xi _i,\xi _j]=\delta _{ij},\; 1\leq i,j\leq n$, uniquely extends to a Poisson bracket on the superalgebra $G(n)=G(n)_{\bar{0}}+ G(n)_{\bar{1}}$.
\end{example}

An associative commutative superalgebra with a Poisson bracket is called a \textit{Poisson superalgebra}.

\vspace{0.8em}

The Poisson superalgebra of Example 2 is denoted as $H_n$.

\vspace{0.8em}

Given two Poisson superalgebras $A,B$, their tensor product is again a Poisson superalgebra: $$[a_1\otimes b_1, a_2\otimes b_2]=(-1)^{|b_1|\cdot |a_2|}([a_1,a_2]\otimes b_1b_2 + a_1a_2\otimes [b_1,b_2]).$$

I. Kantor \cite{7Kantor1990ConnectPJL} noticed that if $A$ is a Poisson superalgebra then the vector space $J=A+Av$ with the operation that extends the multiplication in $A$ and $a(bv)=abv,\; (bv)a=(-1)^{|a|}bav,\; (av)(bv)=(-1)^{|b|}[a,b];\; a,b\in A_{\bar{0}}\cup A_{\bar{1}}$, is a Jordan superalgebra.
We call it the \textit{Kantor double} of the bracket $[\; ,\; ]$ and denote it as $K(A, [\; ,\; ])$ or simply $K(A)$.

\vspace{0.8em}

There exists, however, non-Poisson brackets whose Kantor doubles are Jordan superalgebras. We call such brackets \textit{Jordan brackets}.
D. King and K. McCrimmon \cite{8Daniel1992TheKC} characterized Jordan brackets in terms of identities.

\begin{example}
    Let $A$ be an associative commutative superalgebra with an even derivation $D$. Then the bracket $[a,b]=D(a)b-aD(b)$ is Jordan, though not Poisson. We say that $[a,b]$ is a bracket of \textit{vector type}.
\end{example}
We notice that if $A$ is an associative commutative superalgebra with the identity element $1$ and $[\; ,\; ]$ is a Jordan bracket on $A$, then $a'=[a,1]$ is an even derivation on $A$. For arbitrary elements $a,b,c\in A_{\bar{0}}\cup A_{\bar{1}}$, we have $$[ab,c]=a[b,c]+(-1)^{|b|\cdot |c|}[a,c]b+abc' .$$
\begin{lemma}
    Let $A=A_{\bar{0}}+ A_{\bar{1}}$ be an associative commutative superalgebra with a Jordan bracket. Then there exists a unique Jordan bracket on $A\otimes G(n)$ that
    \begin{enumerate}
        \item extends the Jordan bracket on $A$,
        \item extends the Poisson bracket (see Example 3) on $G(n)$,
        \item $[\xi _1\cdots\xi _k,a]=(k-1)\xi _1\cdots\xi _k a'$ for an arbitrary element $a\in A$.
    \end{enumerate}
\end{lemma}
\begin{proof}
    Straightforward computation.
\end{proof}

\begin{definition}
    A binary bilinear product $[\; ,\; ] : A\times A\to A$ is called a \textit{contact bracket} if
    \begin{enumerate}
        \item $(A, [\; ,\; ])$ is a Lie superalgebra,
        \item the linear transformation $D:a\mapsto [a,1], a\in A$, is an even derivation of $A$,
        \item $[ab,c] = a[b,c] + (-1)^{|b|\cdot |c|}[a,c]b + ab D(c)$ for arbitrary elements $a,b,c\in A_{\bar{0}}\cup A_{\bar{1}}$.
    \end{enumerate}
\end{definition}

N. Cantarini and V. Kac \cite{2Cantarini2007Classification} noticed that Jordan brackets are in 1-1 correspondence with contact brackets. More precisely, if $[a,b]$ is a contact bracket with derivation $D(a)=[a,1]$, then $\left \langle a,b \right \rangle = [a,b]-\frac{1}{2}(D(a)b-aD(b))$ is a Jordan bracket. Even derivations corresponding to the brackets $[\; ,\; ],\left \langle \; ,\; \right \rangle $ are different: $\left \langle a,1 \right \rangle = \frac{1}{2}[a,1]$.

An associative commutative superalgebra $A$ with a contact bracket is called a \textit{contact algebra}.

Now we are ready to formulate the analog of Lemma 7 for contact algebras.
\begin{lemma}
    Let $A$ be a contact algebra. Then there exists a unique contact bracket on $A\otimes G(n)$ that
    \begin{enumerate}
        \item extends the contact bracket on $A$,
        \item extends the Poisson bracket on $G(n)$,
        \item $[\xi _1\cdots\xi _k,a]=\frac{k-2}{2}\xi _1\cdots\xi _k a'$ for an arbitrary element $a\in A$.
    \end{enumerate}
\end{lemma}
\begin{remark}
    The lemma above is a partial answer to the Question 1 from \cite{11Zelmanov2019BSA}.
\end{remark}

Let $A$ be an associative commutative superalgebra with a contact bracket $[a,b]$. Following N. Cantarini and V. Kac \cite{2Cantarini2007Classification}, we define the Jordan bracket $\left \langle a,b \right \rangle = [a,b]-\frac{1}{2}(D(a)b-aD(b))$. Extend the Jordan bracket $\left \langle a,b \right \rangle$ to a Jordan bracket on $A\otimes G(n)$ as in Lemma 7. Extend the contact bracket $[a,b]$ to a contact bracket on $A\otimes G(n+3)$ as in Lemma 8.
Let $L=(A\otimes G(n+3),[\; ,\; ])$.

\vspace{0.8em}

Then (refn. \cite[Chapter 6]{2Cantarini2007Classification}) $$\operatorname{TKK}(K(A\otimes G(n),\left \langle \; ,\; \right \rangle ))\cong [L,L] .$$

\section{Cyclic homology of Kantor doubles}
\label{CHofKantorDoubles}
Let $A$ be an associative commutative superalgebra with a Jordan bracket $[\; ,\; ]$. Let $J=K(A, [\; ,\; ])$ be the Kantor double, $J=A+Av,\; J_{\bar{0}}=A_{\bar{0}}+ A_{1}v,\; J_{\bar{1}}=A_{\bar{1}}+ A_{0}v$.

\subsection{Poisson centers}
\label{Poissoncenters}

\vspace{0.5em}

Let $A$ be an associative commutative superalgebra with a Jordan bracket $[x,y]$. We define the \textit{Poisson center} of $A$ as $Z_p=\{ u\in A\; |\; u'=0, (cu)'+[c,u]=0, \forall c\in A \}$.
If $[x,y]$ is a Poisson bracket then $Z_p$ is the well known Poisson center.

\vspace{0.7em}

Let $J = A + Av$ the Kantor double Jordan algebra and $\lambda : A\to F$ be a linear functional. Define a bilinear mapping $J\times J\to F,\;\; x\times y\mapsto (x\; |\; y)_{\lambda }$ via $(A\; |\; A)_{\lambda }=(0)$,$(Av\; |\; A)_{\lambda }=(0)$, $(av\; |\; bv)_{\lambda }=(-1)^{|b|}\lambda (ab);\; a,b\in A$.

\begin{lemma}
(1) For an arbitrary linear functional $\lambda\in A^{*}$, $(x\; |\; y)_{\lambda }$ is a cyclic cocycle of $J$;
\\
(2) $(x\; |\; y)_{\lambda }$ is a coboundary if and only if $\lambda (Z_p)=(0)$.
\end{lemma}
\begin{proof}
    The assertion (1) is straightforward.
    Let us check the assertion (2).
    
\vspace{0.8em}

Choose elements $a,b \in A_{\bar{0}} \cup A_{\bar{1}}$ and consider the inner derivation $D(av,bv)=R(av)R(bv)-(-1)^{(|a|+1)(|b|+1)}R(bv)R(av)$ of the Jordan superalgebra $J=\operatorname{Kan}(A)$. For an arbitrary element $c\in A$ we have
    $$\begin{aligned}
    cD(av,bv) &= cav\cdot bv + (-1)^{|a||b|+|a|+|b|}cbv\cdot av \\
&= (-1)^{|b|}[ca,b]+(-1)^{|a||b|+|b|}[cb,a] \\
&= (-1)^{|b|+|a||b|}[c,b]a+(-1)^{|b|}c[a,b]+(-1)^{|b|}cab'\\ &\hskip .4cm +(-1)^{|b|}[c,a]b+(-1)^{|b|+|a||b|}c[b,a] +(-1)^{|a||b|+|b|}cba' \\
&=(-1)^{|b|}((-1)^{|a||b|}[c,b]a+[c,a]b)+(-1)^{|b|}c(ab)' . 
    \end{aligned}$$
    On the other hand, $$[c,ab]=[c,a]b+(-1)^{|a||b|}[c,b]a-c'ab.$$
    Hence $$cD(av,bv)=(-1)^{|b|}([c,ab]+(cab)') .$$
    Furthermore, $$vD(av,bv)=(-1)^{|a|}[1,a]bv+(-1)^{|a||b|+|a|}[1,b]av=(-1)^{|a|}[1,ab]v=-(-1)^{|b|}v(ab)'.$$
    This implies that $\sum _i D(a_iv,b_iv)=0$ if and only if $\sum _i (-1)^{|b_i|}a_ib_i\in Z_p$.

\vspace{0.8em}

   Suppose that a cyclic cocycle $(x\; |\; y)_{\lambda }$ is a coboundary. Then there exists a linear functional $\mu : \operatorname{Inder}(J)\to F$, such that $(x\; |\; y)_{\lambda } = \mu (D(x,y))$. Let $u\in Z_p$. Then $D(uv,v)=0$, which implies $(uv\; |\; v)_{\lambda }=\lambda (u)=\mu (0) = 0$. We proved that $\lambda (Z_p)=(0)$.

\vspace{0.8em}

    Conversely, suppose that $\lambda (Z_p)=(0)$. Define $\mu : D(Av,Av) \to F$ via $\mu (D(av,bv))=(-1)^{|b|}\lambda (ab)$. To show that this mapping is well-defined, we need to verify that $\sum _i D(a_iv,b_iv)=0$ implies $\sum _i(-1)^{|b_i|}\lambda (a_ib_i)=0$. But we know that $\sum _i D(a_iv,b_iv)=0$ if and only if  $\sum _i (-1)^{|b_i|} a_ib_i \in Z_p$.  And  $\sum _i (-1)^{|b_i|}  a_ib_i \in Z_p$ implies that $\sum _i(-1)^{|b_i|}\lambda (a_ib_i)=0$ by our assumption.

\vspace{0.8em}

    Extend $\mu$ to a mapping $\operatorname{Inder}(J)\to F$ via $\mu (D(A,A))=\mu (D(Av,A)) = (0)$. Hence the cocycle  $(x|y)_{\lambda}$ is a coboundary. This completes the proof of the lemma.
\end{proof}

\subsection{Bracket cyclic cocycles}
\label{BracketCyclicCocycle}

\vspace{0.5em}

Let $(x\; |\; y)$ be a cyclic cocycle on $J=\operatorname{Kan}(A)$. Consider the linear functional $\lambda : A\to F$ defined via $\lambda (a)=(av\; |\; v)$. Now,
\begin{equation}
\label{eq7}
    (av\; |\; bv) = (a\;|\; v\cdot bv) - (-1)^{|a|}(v\; |\; a\cdot bv) = (-1)^{|b|+1}(a\; |\; b') + (-1)^{|b|+1}\lambda (ab).
\end{equation}

It follows that
\begin{equation}
\label{eq8}
    (a'\; |\; b) = (-1)^{|a|\cdot |b|}(b'\; |\; a).
\end{equation}

Let us explore the cocycle condition for $av,bv,c;\; a,b,c\in A_{\bar{0}}\cup A_{\bar{1}}$:
$$(av\cdot bv\; |\; c) + (-1)^{(|a|+1)(|b|+|c|+1)}(bv\cdot c\; |\; av) + (-1)^{|c|(|a|+|b|)}(c\cdot av\; |\; bv) = 0.$$

In view of \eqref{eq7}, it is equivalent to
\begin{equation}
\label{eq9}
    ([a,b]\; |\; c) = (a'\; |\; bc) - (-1)^{|a|\cdot |b|}(b'\; |\; ac).
\end{equation}

\begin{definition}
    We call a cyclic cocycle $(x\; |\; y)$ a \textit{bracket cyclic cocycle} if \eqref{eq8} and \eqref{eq9} hold.
\end{definition}

\begin{example}
    $(f\; |\; g) = \operatorname{Res}f'g$ is a bracket cyclic cocycle on $F[t,t^{-1}]$ relative to the bracket $[f,g]=f'g-fg'$.
\end{example}

We showed that if $(x\; |\; y)$ is a cyclic cocycle on $\operatorname{Kan}(A)$, then \begin{enumerate}
    \item the restriction of $(x\; |\; y)$ to $A\times A$ is a bracket cyclic cocycle,
    \item $(av\; |\; bv) = (-1)^{|b|+1}(a\; |\; b') -(a\; |\; b)_{\lambda }.$
\end{enumerate}

\begin{lemma}
    Let $(x\; |\; y)$ be a bracket cyclic cocycle on $A$; let $\lambda\in A^{*}, J=\operatorname{Kan}(A)=A+Av$. Then the super skew-symmetric mapping $J\times J\to F$ that
    \begin{enumerate}
        \item extends $(x\; |\; y)$ on $A\times A$,
        \item $(av\; |\; bv) = (-1)^{|b|+1}(a\; |\; b') -(a\; |\; b)_{\lambda }$,
        \item $(Av\; |\; A)=(0)$,
    \end{enumerate}
    is a cyclic cocycle on $J$.
\end{lemma}
\begin{proof}
    Choose three elements for the cocycle identity. If all three elements lie in $A$ or two elements of them lie in $Av$, then the identity follows from the computation above. If one or three elements lie in $Av$, then the identity follows from $(Av\; |\; A)=(0)$.
\end{proof}

Let $C_{br}(A)$ be the vector space of all bracket cyclic cocycles $A\times A\to F$.

\vspace{0.8em}

Each bracket cyclic cocycle on $A$ can be extended to a cyclic cocycle on $J$ via $(av\; |\; bv) = (-1)^{|b|}(a\; |\; b'), (Av\; |\; A)=(0)$. Hence $C_{br}(A)\subseteq C(J)$.

\vspace{0.8em}

For arbitrary elements $a,b\in A$, the inner derivation $D(a,b)$ of $J$ is zero. Hence $C_{br}(A)\cap B(J)=(0)$ and therefore $HC_{br}(A)=C_{br}(A)\subseteq HC(J)$.

\subsection{Mixed cocycles}
\label{MixedCocycle}

\vspace{0.5em}

\begin{definition}
    We call a cyclic cocycle $(x\; |\; y)$ on $J$ \textit{mixed} if $(A\; |\; A)=(Av\; |\; Av)=(0)$.
\end{definition}

\begin{lemma}
    Let $\left \langle \; |\; \right \rangle :A\times A\to F$ be a bilinear mapping. The mapping $(\; |\; ) : J\times J\to F, (a\; |\; bv) = \left \langle a\; |\; b\right \rangle , (bv\; |\; a) = -(-1)^{|a|(|b|+1)} \left \langle a\; |\; b\right \rangle , (A\; |\; A)=(Av\; |\; Av)=(0)$ is a cyclic cocycle if and only if

\vspace{0.8em}

    (i) $\left \langle [a,b]\; |\; c\right \rangle + (-1)^{|a|(|b|+|c|)}\left \langle [b,c]\; |\; a\right \rangle + (-1)^{|c|(|a|+|b|)}\left \langle [c,a]\; |\; b\right \rangle = 0$,

    (ii) $\left \langle a\; |\; bc\right \rangle + (-1)^{|a|(|b|+|c|)}\left \langle b\; |\; ca\right \rangle + (-1)^{|c|(|a|+|b|)}\left \langle c\; |\; ab\right \rangle = \left \langle abc\; |\; 1\right \rangle$,
    
\vspace{0.8em}

for arbitrary elements $a,b,c\in A_{\bar{0}}\cup A_{\bar{1}}$.
\end{lemma}
\begin{proof}
    The bilinear mapping $(x\; |\; y)$ is super skew-symmetric by definition. We need to examine the cyclic cocycle identity for the following two triples:
    \begin{enumerate}
        \item $av,bv,cv$;
        \item $av,b,c$,
    \end{enumerate}
    where $a,b,c\in A_{\bar{0}}\cup A_{\bar{1}}$.

    Again to simplify computation, without loss of generality, we will assume that $A=A_{\bar{0}}$.
    Then 
		
		 \begin{equation}
        \label{eq10}
        \begin{aligned}
    \quad & \quad (av\cdot b \; |\; c) + (bc \; |\; av) + (c\cdot av \; |\; b) \\
&= -(c \; |\; abv) + (bc \; |\; av) - (b \; |\; cav) \\
&= - \left \langle c\; |\; ab\right \rangle + \left \langle bc\; |\; a\right \rangle - \left \langle b\; |\; ca\right \rangle .
    \end{aligned}
    \end{equation}
		
		Furthermore,
		$$\begin{aligned}
    \quad & \quad (av\cdot bv \; |\; cv) + (bv\cdot cv \; |\; av) + (cv\cdot av \; |\; bv) \\
&= ([a,b] \; |\; cv) + ([b,c] \; |\; av) + ([c,a] \; |\; bv) \\
&= \left \langle [a,b]\; |\; c\right \rangle + \left \langle [b,c]\; |\; a\right \rangle + \left \langle [c,a]\; |\; b\right \rangle .
    \end{aligned}$$
    So, the cocycle identity is satisfied for elements $av, \, bv, \; cv$, $a,b,c \in A$ if and only if the identity 
		(i) is satisfied.
		
\medskip

    Suppose that $(x\; |\; y)$ is a cyclic cocycle. Then
		
		$$(av\cdot b \; |\; c) + (bc \; |\; av) + (c\cdot av \; |\; b) = 0.$$
		
		Doing $b = c = 1$ in \eqref{eq10} we get
		
    \begin{equation}
        \label{eq11}
        \left \langle 1\;|\; a \right \rangle = (1\; |\; av) = 0, a\in A.
    \end{equation}
    
    Applying the cocycle identity to elements $a,b,v$, we get $(a\; |\; bv) + (b\; |\; va) + (v\; |\; ab) = 0$, which implies
    \begin{equation}
        \label{eq12}
        \left \langle a\;|\; b \right \rangle + \left \langle b\;|\; a \right \rangle = \left \langle ab\;|\; 1 \right \rangle .
    \end{equation}

    Applying \eqref{eq12} to the right hand side of \eqref{eq10} we get
		$$\left \langle ab\;|\; c \right \rangle - \left \langle abc\;|\; 1 \right \rangle + \left \langle bc\;|\; a \right \rangle + \left \langle ca\;|\; b \right \rangle - \left \langle abc\;|\; 1 \right \rangle =0.$$
		
		That is,
		$$\left \langle ab\;|\; c \right \rangle + \left \langle bc\;|\; a \right \rangle + \left \langle ca\;|\; b \right \rangle = 2 \left \langle abc\;|\; 1 \right \rangle .$$

    Again applying \eqref{eq12} to each summand on the identity above, we get 
		$$- \left \langle c\;|\; ab \right \rangle - \left \langle a\;|\; bc \right \rangle - \left \langle b\;|\; ca \right \rangle + 3 \left \langle abc\;|\; 1 \right \rangle = 2 \left \langle abc\;|\; 1 \right \rangle ,$$
    which implies (ii).

\vspace{0.8em}

    Now suppose that $\left \langle x\;|\; y \right \rangle$ satisfies the identities (i) and (ii). As mentioned above, the identity (i) implies that the cocycle identity is satisfied for elements $av, \, bv, \; cv$, $a,b,c \in A$. Substituting $b=c=1$ in (ii), we get \eqref{eq11}. Substituting $c=1$ in (ii), we get \eqref{eq12}.

\vspace{0.8em}

    Now the identity (ii) together with \eqref{eq11}, \eqref{eq12} imply that the last line of \eqref{eq10} is equal to $0$, hence the cyclic cocycle identity holds for elements $av,b,c$. This completes the proof that $(x\; |\; y)$ is a cyclic cocycle.
\end{proof}

\noindent\textbf{Question.} Which mixed cocycles are coboundaries?

\vspace{0.8em}

Consider the linear mapping $$\mu : A\otimes_F A \mapsto A,\;\; (a \otimes b)\mapsto [a,b]+(-1)^{|a|\cdot |b|}b'a.$$
Denote $W=\operatorname{Ker}\mu$.
\begin{lemma}
    \begin{enumerate}
        \item $\sum _i D(a_i,b_iv)=0$ if and only if $\sum _i a_i\otimes b_i\in W$;
        \item let $\left \langle \; | \; \right \rangle : A\times A\to F$ be a bilinear mapping satisfying i) and ii) in Lemma 11. We can identify it with $\left \langle \; | \; \right \rangle : A\otimes A\to F$, the unique linear map that defines.  The cyclic cohomology $(a\; |\; bv) = \left \langle a\;|\; b \right \rangle ,  (A\; |\; A)=(Av\; |\; Av)=(0); a,b\in A$ is a coboundary if and only if $\left \langle \; ,\; \right \rangle$ vanishes on $W$.
    \end{enumerate}
\end{lemma}
\begin{proof}
    We have $AD(A,Av)=(0)$. Hence $\sum _i D(a_i,b_iv)=0$ if and only if $v\sum _i D(a_i,b_iv)=0$.

\vspace{0.8em}

    Let $a,b\in A_{\bar{0}}\cup A_{\bar{1}}$. Then $$\begin{aligned}
    vD(a,bv) &= (va)(bv) - (-1)^{|a|(|b|+1)}(v(bv))a  \\
&= (-1)^{|a| + |b|} [a,b] + (-1)^{|a|\cdot |b| + |a| + |b|} b'a\\
&= (-1)^{|a| + |b|} ([a,b] + (-1)^{|a|\cdot |b|} b'a) .
    \end{aligned}$$
    It implies the assertion 1 of the lemma. The assertion 2 directly follows from the assertion 1, arguing as in the Proof of Lemma 6.
\end{proof}

Let us summarize the obtained results.

\subsection{Section Summary}
\label{S4Summary}

\vspace{0.5em}

Let $Z_p$ be the Poisson center of an associative commutative superalgebra with a Jordan bracket and $J = Kan(A)$. For an arbitrary linear functional $\lambda : Z_p\to F$, consider an extension $\widetilde{\lambda } : A\to F$. By Lemma 9, the cohomology class $(x\; |\; y)_{\widetilde{\lambda }} + B(J)$ does not depend on a choice of  an extension $\widetilde{\lambda }$.

\vspace{0.8em}

The vector space $HC_Z(J)=\{ (x\; |\; y)_{\widetilde{\lambda }} + B(J)\; |\; \lambda\in Z_p^{*} \}$ can be identified with the dual space $ Z_p^{*}$.

\vspace{0.8em}

Recall that $C_{br}(A)$ denotes the vector space of bracket cyclic cocycle on the algebra $A$. An arbitrary bracket cyclic cocycle $(x\; |\; y)$ extends to a cyclic cocycle on $J$ via $(av\; |\; bv) = (-1)^{|a|\cdot |b|}(b'\; |\; a), (Av\; |\; A)=(0)$. This defines an embedding of $C_{br}(A)$ into the vector space $C(J)$ and, furthermore, into the cyclic homology space $HC(J)$.

\vspace{0.8em}

Let $C_M(J)$ denote the vector space of mixed cyclic cocycles on $J$, $$HC_M(J)=(C_M(J)+B(J))/B(J).$$
\begin{theorem}
    $HC(J)=HC_Z(J)\oplus C_{br}(A)\oplus HC_M(J)$.
\end{theorem}

\section{Poisson brackets and Hamiltonian superalgebras}
\label{PoissonBracketsHamiltonianSuperalgebra}
In this section we compute cyclic homology of Kantor doubles $K(A)$, where $A=(A,[\; ,\; ])$ is a Poisson superalgebra.

\vspace{0.8em}

Recall that $H_n=(F[p_1,\dots ,p_n,q_1,\dots ,q_n],[p_i,q_j]=\delta _{ij},[p_i,p_j]=[q_i,q_j]=0),n\geq 1$ is a family of Poisson algebras.

\vspace{0.8em}

Let $B$ be an arbitrary Poisson superalgebra and let $A=H_1\otimes B$.
\begin{theorem}
    $HC(A)= Z_p(B)^{*}$, where $Z_p(B)$ denotes the Poisson center of $B$.
\end{theorem}

Consider the Grassmann superalgebra $G(3)$ with the standard Poisson bracket and the tensor product $A\otimes G(3)$ of Poison superalgebras.
Abusing notation we denote the Poisson bracket on $A\otimes G(3)$ as $[\; ,\; ]$.
Consider the Lie superalgebra $L=(A\otimes G(3), [\; ,\; ])$.
The superalgebra $[L,L]$ is isomorphic to the (universal) Tits-Kantor-Koecher construction of the Jordan superalgebra $K(A)$.
Theorems 1, 2, 3 yield description of the vector space $H^2([\bar{L},\bar{L}])=H^2(\overline{\operatorname{TKK}}(K(A)))$, where $\overline{L}$ is the quotient of the superalgebra $L$ modulo the center.
In particular we compute second homologies of Hamiltonian superalgebras.

\vspace{0.8em}

Let $A=H_n\otimes G(m),n\geq 1,m\geq 3$. Let $H(n,m)$ denote the Lie superalgebra $(A,[\; ,\; ])$.
\begin{theorem}
    $\dim _F H^2(H(n,m)) = 1$.
\end{theorem}

\textit{Throughout this section we assume that $A=H_1\otimes _F B$, where $B$ is a Poisson superalgebra and $H_1=(F[p,q],[p,q]=1)$.}
\begin{lemma}
    $C_{br}(A)=(0)$.
\end{lemma}
\begin{proof}
    If $(a\; |\; b)$ is a bracket cyclic cocycle on $A$, then the identity \eqref{eq9} implies that $([A,A]\; |\; A)=(0)$.
    It remains to show that $[A,A]=A$.
    Indeed, the equality $p^iq^j=\frac{1}{j+1}[p,p^iq^{j+1}]$ implies that $[H_1,H_1]=H_1$. Consider an arbitrary tensor $h\otimes b, h\in H_1, b\in B$. Suppose that $h=\sum _i [h_i',h_i'']$. Then $h\otimes b=\sum _i [h_i'\otimes b,h_i''\otimes 1]$.
\end{proof}

\begin{lemma}
    $Z_p(H_1\otimes B)=Z_p(B)$.
\end{lemma}
\begin{proof}
    The inclusion $1\otimes Z_p(B) \subseteq Z_p(H_1\otimes B)$ is straightforward.

\vspace{0.8em}

    Suppose that an element $c=\sum _i h_i\otimes b_i$ lies in $Z_p(H_1\otimes B);h_i\in H_1, b_i\in B$, the elements $\{ b_i \} _i$ are linearly independent. For an arbitrary element $h\in H_1$ we have $$[h\otimes 1,\sum _i h_i\otimes b_i]=\sum _i [h,h_i]\otimes b_i,$$ which implies that all elements $[h,h_i]$ are equal to zero, $h_i\in Z_p(H_1)=F\cdot 1$. Hence $c=1\otimes b,b\in B$. Again for an arbitrary element $b'\in B$ we have $[1\otimes b',1\otimes b]=1\otimes [b',b]=0$. Hence $b\in Z_p(B)$.
\end{proof}

Now our aim is
\begin{proposition}
    $HC_M(K(A))=(0)$.
\end{proposition}

Proof of this proposition requires several lemmas.

\begin{lemma}
    A mixed cyclic cocycle $(~ |~)$ on $K(A)$ is a coboundary if and only if $\sum _i [a_i,b_i]=0; a_i,b_i\in A_{\bar{0}}\cup A_{\bar{1}}$ implies $\sum _i \left \langle a_i\; |\; b_i \right \rangle = 0$. Here $\left\langle ~ |~ \right \rangle$ is the bilinear map defining $(~ |~)$.
\end{lemma}
This lemma immediately follows from Lemma 12 (2).
\begin{lemma}
    Let $(x\; |\; y)$ be a cyclic cocycle on $K(A)$. Then $(A\; |\; v)=(0)$.
\end{lemma}
\begin{proof}
    We showed in the proof of Lemma 13 that $A=[A,A]=(Av)(Av)$. By the cyclic cocycle identity $(av\cdot bv\; |\; v) + (-1)^{|b|(|a|+1)}(bv\cdot v\; |\; av) + (-1)^{|a|+|b|}(v\cdot av\; |\; bv)=0$. Since $[\; ,\; ]$ is a Poisson bracket, it follows that $bv\cdot v=[b,1]=-b'=0,v\cdot av=(-1)^{|a|}[1,a]=0$.
\end{proof}

Let $(x\; |\; y)$ be a mixed cocycle on $K(A)$. As above let $\left \langle a\; |\; b \right \rangle = (a\; |\; bv); a,b\in A$. By Lemma 16 we have $\left \langle A\; |\; 1 \right \rangle = (0)$.

\vspace{0.8em}

Now the identity \eqref{eq12} and Lemma 11 imply that
\begin{equation}
    \label{eq13}
    \left \langle a\; |\; b \right \rangle = - (-1)^{|a|\cdot |b|}\left \langle b | a \right \rangle ,
\end{equation}
\begin{equation}
    \label{eq14}
    \left \langle a\; |\; bc \right \rangle + (-1)^{|a|(|b|+|c|)}\left \langle b\; |\; ca \right \rangle + (-1)^{|c|(|a|+|b|)}\left \langle c\; |\; ab \right \rangle = 0,
\end{equation}
\begin{equation}
    \label{eq015}
    \left \langle [a,b]\; |\; c \right \rangle + (-1)^{|a|(|b|+|c|)}\left \langle [b,c]\; |\; a \right \rangle + (-1)^{|c|(|a|+|b|)}\left \langle [c,a]\; |\; b \right \rangle = 0.
\end{equation}

\vspace{0.8em}

The identity \eqref{eq015} repeats the first identity of Lemma 11 for convenience of a reader.

\vspace{0.8em}

Consider the tensor square $A\otimes A$ and the subspace $S\subseteq A\otimes A$ spanned by  elements $a\otimes b + (-1)^{|a|\cdot |b|}b\otimes a,ab\otimes c + (-1)^{|b|\cdot |c|}ac\otimes b + (-1)^{|a|(|b|+|c|)}bc\otimes a, [a,b]\otimes c + (-1)^{|b|\cdot |c|}[a,c]\otimes b  + (-1)^{|a|(|b|+|c|)}[b,c]\otimes a; a,b,c\in  A_{\bar{0}}\cup A_{\bar{1}}$.

\vspace{0.8em}

\begin{lemma}
    $A\otimes A = p\otimes A + S$.
\end{lemma}
\begin{proof}
    We will prove the lemma in several steps.
    \begin{enumerate}
        \item We notice that $A=[p,A]$. Indeed, $p^iq^jb=\frac{1}{j+1}[p,p^iq^{j+1}b]$ for an arbitrary element $b\in B$. Similarly, $A=[q,A]$.
        \item For an arbitrary element $a\in A$, $1\otimes a\in S$.
        \item For an arbitrary element $a\in A$ we have $[p,q]\otimes a=p\otimes [q,a]-q\otimes [p,a]\bmod S$, which implies that
        \begin{equation}
            \label{eq15}
            p\otimes [q,a] - q\otimes [p,a] \in S.
        \end{equation}
        In view of Step 1 we conclude that $p\otimes A= q\otimes A\bmod S$.
        \item Consider the linear operator $P: A\to A, a\mapsto a+q[p,a]$.
        This operator is a bijection. Indeed, if $a=p^iq^jb,b\in B$, then $P(a)=(1+j)p^iq^jb$.
        \item Let $a\in A,b\in B$, then $$
    b\otimes a = [p,qb]\otimes a =  p\otimes [qb,a] - (-1)^{|a|\cdot |b|}qb\otimes [p,a]\bmod S,
    $$
    $$qb\otimes [p,a] = q\otimes b[p,a] + b\otimes q[p,a]\bmod S,$$
    $$b[p,a]=[p,ba],$$
    $$q\otimes b[p,a] = q\otimes [p,ba] = p\otimes [q,ba]\bmod S.$$
    Hence, $$\begin{aligned}
    b\otimes a &= p\otimes [qb,a] - (p\otimes [q,ba] + b\otimes q[p,a]) \bmod S \\
&= p\otimes ([q,ba] - [q,ba]) - b\otimes q[p,a] \bmod S .
    \end{aligned}$$
    Now, 
    $$b\otimes (a+q[p,a]) = p\otimes ([qb,a]-q[b,a]) \bmod S,$$
    $$[qb,a]-q[b,a]=(-1)^{|a|\cdot |b|}[q,a]b=b[q,a],$$
    $$b\otimes P(a)=p\otimes b[q,a]\bmod S.$$
    Finally, $b\otimes a=p\otimes b[q,P^{-1}(a)]\bmod S$.
    \item Obviously, $A\otimes A= p\otimes A + q\otimes A + B\otimes A \bmod S = p\otimes A \bmod S$ by Step 3 and Step 5.

    This completes the proof of the lemma.
    \end{enumerate}
\end{proof}

\begin{proof}[Proof of the Proposition 1]

    Let $(x\; |\; y)$ be a mixed cocycle on $K(A)$. Let us check the condition of Lemma 15, that is, $\sum _i [a_i,b_i]=0; a_i,b_i\in A_{\bar{0}}\cup A_{\bar{1}}$ implies $\sum _i \left \langle a_i\; |\; b_i \right \rangle = 0$.

\vspace{0.8em}

    By Lemma 17, there exists an element $c\in A$ such that $\sum _i a_i\otimes b_i = p\otimes c \bmod S$. If $\sum _i x_i\otimes y_i \in S$, then $\sum _i [x_i, y_i] = 0$. By \eqref{eq13}, \eqref{eq14} and \eqref{eq015}, we also have $\sum _i \left \langle x_i\; |\; y_i \right \rangle = 0$. Hence we need to show that $[p,c]=0$ implies $\left \langle p\; |\; c \right \rangle = 0$.

\vspace{0.8em}

     The equality $[p,c]=0$ implies $c=\sum _{i\geq 0} p^ib_i,b_i\in B$.
     Let us show that $\left \langle p\; |\; p^ib \right \rangle = 0$. We have $p=\frac{1}{2}[p^2,q]$. Hence, $\left \langle p\; |\; p^ib \right \rangle = \frac{1}{2} \left \langle [p^2,q]\; |\; p^ib \right \rangle = \frac{1}{2} ( \left \langle p^2\; |\; [q,p^ib] \right \rangle - \left \langle q\; |\; [p^2,p^ib] \right \rangle )$.

\vspace{0.8em}

     If $i\geq 1$, then $[q,p^ib] = -ip^{i-1}b$. If $i=0$, then $[q,b]=0$. We also have $[p^2,p^ib]=0$. Assuming $i\geq 1$, we have $\left \langle p\; |\; p^ib \right \rangle = -\frac{i}{2} \left \langle p^2\; |\; p^{i-1}b \right \rangle = -i \left \langle p\; |\; p^{i}b \right \rangle $ and $(i+1)\left \langle p\; |\; p^ib \right \rangle = 0$. Hence $\left \langle p\; |\; p^ib \right \rangle = 0$.
\end{proof}

\section{Brackets of vector type}
\label{BracketVT}
Let $A$ be an associative commutative superalgebra with an even derivation $' : A\to A$.
Consider a Jordan bracket $[a,b]=a'b-ab'$.
We make the additional assumption that $1\in A_{\bar{0}}'A_{\bar{0}}$.
We will determine bracket cyclic cocycles and Poisson center of the Kantor double $K(A,[\; ,\; ])$.

\vspace{0.8em}

Let $(a\; |\; b)$ be a bracket cyclic cocycle on $A$. By \eqref{eq9} we have \begin{equation}
    \label{eq17}
    (a'b-ab'\; |\; c) = (a'\; |\; bc) - (-1)^{|a|\cdot |b|}(b'\; |\; ac).
\end{equation}
On the other hand, the cocycle identity implies
$$(a'b\; |\; c) = (a'\; |\; bc) + (-1)^{|a|\cdot |b|}(b\; |\; a'c),$$
$$(ab'\; |\; c) = (a\; |\; b'c) + (-1)^{|a|\cdot |b|}(b'\; |\; ac).$$
Substituting these equalities to the left hand side of \eqref{eq17} we get $(a'\; |\; bc) + (-1)^{|a|\cdot |b|}(b\; |\; a'c) - (a\; |\; b'c) - (-1)^{|a|\cdot |b|}(b'\; |\; ac) = (a'\; |\; bc) - (-1)^{|a|\cdot |b|}(b'\; |\; ac)$, which implies \begin{equation}
    \label{eq18}
    (a\; |\; b'c) = (-1)^{|a|\cdot |b|}(b\; |\; a'c).
\end{equation}

\vspace{0.8em}

We claim that there exists a linear functional $\lambda :A\to F$ such that $(a\; |\; b) = \lambda (a'b)$ for arbitrary elements $a,b\in A$.

\vspace{0.8em}

To prove the claim we need to show that $\sum _i a_i'b_i=0;a_i,b_i\in A$ implies $\sum _i (a_i\; |\; b_i) = 0$. 

By \eqref{eq18} for an arbitrary element $c\in A$, we have $$(c\; |\; a_i'b_i) = (-1)^{|a_i|\cdot |c|}(a_i\; |\; c'b_i).$$  
Hence $\sum _i (-1)^{|a_i|\cdot |c|}(a_i\; |\; c'b_i) = 0$.

\medskip

If $\sum _i a_i'b_i=0$, then for an arbitrary element $x\in A$, we have $\sum _i a_i'b_ix=0$. Hence $\sum _i (-1)^{|a_i|\cdot |c|}(a_i\; |\; c'b_ix) = \sum _i (-1)^{|a_i|\cdot |c| + |b_i|\cdot |x|}(a_i\; |\; c'xb_i) = 0$.

\medskip

Since $1\in A_{\bar{0}}'A_{\bar{0}}$, there exist elements $c_j,x_j\in A_{\bar{0}}$ such that $\sum _j c_j'x_j=1$. Now $\sum _{i,j} (a_i\; |\; c_j'x_jb_i) = \sum _i (a_i\; |\; b_i) = 0$.

\medskip

We have proved that $(a\; |\; b) = \lambda (a'b); a,b\in A$ for some linear functional $\lambda\in A^{*}$.

\vspace{0.8em}

For an arbitrary element $a\in A$, $(a\; |\; 1) = \lambda (a') = 0$. Hence $\lambda (A') = (0)$.
It is easy to see that if $\lambda$ is a linear functional on $A/A'$, then $(a\; |\; b) = \lambda (a'b)$ is a cyclic cocycle on $A$. Hence $C_{br}(A)$ can be identified with the dual space $(A/A')^{*}$.

\vspace{0.8em}

If $a$ lies in the Poisson center of $A$, then $a'=0$ and for an arbitrary element $b\in A$, we have $[b,a]+(ba)'=2b'a=0$. Again from $1\in A_{\bar{0}}'A_{\bar{0}}$, it follows that $a=0$, so $Z_p(A)=(0)$.

\vspace{0.8em}

{\bf Remark.}  We don't have a description of the space of mixed cyclic cocycles $HC_M(K(A))$.

\section{Superalgebras $A\otimes G(n),n\geq 1$}
\label{SuperalgebraAG}
Let $A$ be an associative commutative superalgebra with a Jordan bracket. Let $G(n)= \left \langle 1,\xi _1,\dots ,\xi _n \; |\; \xi _i \xi _j + \xi _j \xi _i = 0 \right \rangle $ be the Grassmann algebra with a standard Poisson bracket. We have already mentioned above that both brackets on $A$ and $G(n)$ extend to a Jordan bracket on $A\otimes G(n)$ via $$[A, \xi _i]=(0);\; \xi _i'=0,\; 1\leq i\leq n.$$

\vspace{0.8em}

\textit{In this section we assume that the bracket on $A$ is of vector type $[a,b]=a'b-ab'$, and $1\in A_{\bar{0}}'A_{\bar{0}}$.}

\vspace{0.8em}

Let us start with $n=1$. 

Let $(x\; |\; y)$ be a bracket cyclic cocycle on $A\otimes G(1)$, $G(1)=F\cdot 1+F\xi$, $\xi ^2=0$, $[\xi ,\xi ]=1$.

As  shown in the previous section, there exists a linear functional $\lambda \in (A/A')^{*}$ such that $(a\; |\; b) = \lambda (a'b);\; a,b\in A$. Choose $a\in A, x\in A\cup A\xi $. Then $a=[a\xi , \xi ]$. Hence $$(a\; |\; x) = ([a\xi , \xi ]\; |\; x) = (a'\xi \; |\; \xi x) = (-1)^{|a|}(\xi \; |\; a'\xi x).$$ This implies that $(A\; |\; A\xi )=(0)$. 

Now let $x=b\in A$. Then $(a\; |\; b) = (-1)^{|a|}(\xi \; |\; a'{\xi}b )$. Since $A'A=A$, it follows that $$(\xi \; |\; a\xi )=(-1)^{|a|}\lambda (a), (\xi a \; |\; \xi b)=(-1)^{|a|}(\xi \; |\; \xi ab)=(-1)^{|a|}\lambda (ab),$$ for arbitrary elements $a,b\in A$.

\vspace{0.8em}

Again we conclude that $C_{br}(A[\xi ])=(A/A')^{*}$.

\vspace{0.8em}

Let us show that $Z_p(A[\xi ])=Z_p (A)$. Indeed, the inclusions $Z_p (A)\subseteq Z_p(A[\xi ])$ and $Z_p(A[\xi ])\cap A\subseteq Z_p (A)$ are obvious. Let $a\in A$ such that $a\xi \in Z_p(A[\xi ])$. Then $a=[a\xi ,\xi ]=0$. This proves the claim.

\vspace{0.8em}

If $[a,b]=a'b-ab'$ is a vector type bracket and $A'A=A$, then $Z_p (A)=(0)$.

\vspace{0.8em}

We have nothing to say about mixed cyclic cocycles of $A$.

\vspace{0.8em}

Now let $n\geq 2$.
\begin{lemma}
    For an arbitrary associative commutative superalgebra $A$ with a Jordan bracket, we have $C_{br}(A\otimes G(n))=(0)$ for $n\geq 2$.
\end{lemma}
\begin{proof}
     Let $(x\; |\; y)$ be a bracket cyclic cocycle on $A\otimes G(n)$. We have $$(\xi _1 \; |\; x )=([\xi _1\xi _2 ,\xi _2 ] \; |\; x )=((\xi _1\xi _2)' \; |\;\xi _2 x )-(\xi _2' \; |\; \xi _1\xi _2 x )=0.$$
     We have shown that $(\xi _i \; |\; A\otimes G(n) )= 0, 1\leq i\leq n$.

\vspace{0.8em}

     For an arbitrary element $a\in A$, we have $a=[a\xi _1 ,\xi _1 ]$. Hence, $$(a\; |\; x) = ([a\xi _1 ,\xi _1 ] \; |\; x ) = (a'\xi _1\; |\; \xi _1 x) + (\xi _1' \; |\; a\xi _1 x) = (a'\xi _1\; |\; \xi _1 x).$$

     Furthermore, $$(a'\xi _1\; |\; \xi _1 x) = (a'\; |\; \xi _1\cdot \xi _1 x) + (\xi _1 \; |\; a'\xi _1 x) = 0+0=0.$$
     We have shown that $(A \; |\; A\otimes G(n) )= (0)$.
\end{proof}

As in the case $n=1$, we get $Z_p(A\otimes G(n))=Z_p(A)$.

\vspace{0.8em}

\begin{lemma}
    Let $L$ be a perfect Lie superalgebra. Suppose that there is $h\in L_{\bar{0}}$ such that $L=\oplus _{i\in \mathbb{Z}}L_i$,  $[h,a_i]=ia_i$, $\forall a_i\in L_i$. Let $(x\; |\; y)$ be a $2$-cocycle on $L$. Then $(L_i\; |\; L_j)=(0)$ whenever $i+j\neq 0$.
\end{lemma}

\begin{proof}
    Let $\widehat{L}$ be a universal central extension of $L$. It is known that $\widehat{L}=L+Z$, where
    $Z=\{ \sum _i a_i\otimes b_i\; |\; \sum _i [a_i,b_i]=0 \} /I$, and $I$ is the linear span of the set $ \{ a\otimes b + (-1)^{|a|\cdot |b|}b\otimes a, [a,b]\otimes c + (-1)^{|a|(|b|+|c|)}[b,c]\otimes a + (-1)^{|c|(|a|+|b|)}[c,a]\otimes b\; |\; a,b,c\in L \}$.
    The $\mathbb{Z}$-grading $L=\oplus _{i\in \mathbb{Z}}L_i$ extends to $Z$ and to the superalgebra $\widehat{L}$.

\vspace{0.8em}

    Let us show that $Z_n=(0)$ for $n\neq 0$. Indeed, Let $(x\; |\; y)$ be a $2$-cocycle of $L$ and let $z= \sum _i a_i\otimes b_i, \sum _i [a_i,b_i]=0, \deg (a_i)+\deg (b_i)=n$. 
 
    Since $L$ is perfect, it follows that $h=\sum _j [e_j,f_j];e_j,f_j\in L$. Now, $$\sum _i ([h,a_i ]\; |\; b_i) + \sum _i (a_i\; |\; [h,b_i ]) =$$
		$$\sum _j ([e_j,\sum _i [a_i,b_i]]\; |\; f_j) + \sum _j (e_j\; |\; [f_j,\sum _i [a_i,b_i]]]) = 0.$$
    
    On the other hand, $\sum _i ([h,a_i ]\; |\; b_i) + \sum _i (a_i\; |\; [h,b_i ]) = n\sum _i (a_i\; |\; b_i)$. 
		
		This completes the proof of the lemma.
\end{proof}


The following result shed some light on mixed cyclic cocycles of $A\otimes G(n)$.

Mixed cyclic cocycles on $J=K(A\otimes G(n))$ correspond to $2$-cocycles $(x\; |\; y)$ on Lie superalgebras $(A\otimes G(n+3),[\; ,\; ])$ such that $(a\xi _{i_1}\cdots \xi _{i_p}\; |\; b\xi _{j_1}\cdots \xi _{j_q})\neq 0$,   $\{ i_1,\dots , i_p\}\neq \{ j_1,\dots , j_q\}$. 

For a subset $\pi\subseteq\{ 1,\dots ,n \},\pi =\{ i_1<\dots <i_p \}$, denote $\xi _{\pi } = \xi _{i_1}\cdots \xi _{i_p}$.

\vspace{0.8em}

Let $(x\; |\; y)$ be a $2$-cocycle on a Lie superalgebra $(A\otimes G(n),[\; ,\; ])$, where $A$ is a contact superalgebra, $[\; ,\; ]$ is a contact bracket on $A\otimes G(n)$ that extends the contact bracket on $A$ and the standard Poisson bracket on $G(n)$.

\begin{proposition}
    Let $(a\xi _{\pi }\; |\; b\xi _{\tau })\neq 0;a,b\in A;\pi , \tau\subseteq\{ 1,\dots ,n \} ,\pi\neq \tau$. Then $\pi\dot{\cup}\tau = \{ 1,\dots ,n \}$.
\end{proposition}

\begin{proof}
    Suppose that $k\in \tau\setminus \pi$.
    \begin{enumerate}
        \item Suppose at first that $|\pi\cap \tau |\geq 2$. Then $\pi = \pi ' \dot{\cup} \pi ''$ is a disjoint union of two subsets, such that $\pi '\cap \tau\neq\emptyset$ and $\pi ''\cap \tau\neq\emptyset$. 
				
				We have $[a\xi _{\pi '}\xi _k , \xi _{\pi ''}\xi _k]=\pm a\xi _{\pi }$. Hence 
				$$(a\xi _{\pi }\; |\; b\xi _{\tau }) = \pm ([a\xi _{\pi '}\xi _k , \xi _{\pi ''}\xi _k]\; |\; b\xi _{\tau }) =$$ $$\pm (a\xi _{\pi '}\xi _k\; |\; [\xi _{\pi ''}\xi _k , b\xi _{\tau }]) \pm (\xi _{\pi '}\xi _k\; |\; [a\xi _{\pi '}\xi _k , b\xi _{\tau }]) = 0.$$
        
\vspace{0.8em}

        \item Now suppose that $\pi\cap \tau = \{ i \} $ , $\pi '=\pi\setminus\{ i \}$ , $\tau '=\tau\setminus\{ i,k \} $, $\xi _{\pi }=\pm \xi _i\xi _{\pi '}$, $ \xi _{\tau } = \pm \xi _i \xi _k \xi _{\tau '}$. Let $$\xi = \frac{1}{2}(\xi _i + \sqrt{-1}\xi _k),\eta = \frac{1}{2}(\xi _i - \sqrt{-1}\xi _k),[\xi ,\xi ]=[\eta ,\eta ]=0,[\xi ,\eta ]=\frac{1}{2}.$$
        Then $\xi _i = \xi + \eta , \xi _k = -\sqrt{-1}(\xi - \eta ), h= 2\xi\eta ,[h,\xi ]=\xi ,[h,\eta ]=-\eta.$
        
				The element $a\xi _{\pi }$ is a linear combination of elements $a\xi _{\pi '}\xi $ and $a\xi _{\pi '}\eta $. We also have $b\xi _{\tau } = \pm b\xi _{\tau '}\xi _i \xi _k=\pm 2b\xi _{\tau '}\xi\eta$. Now, $$[h,a\xi _{\pi '}\xi ]=\pm a\xi _{\pi '}\xi, \; [h,a\xi _{\pi '}\eta ] = \pm a\xi _{\pi '}\eta,\; [h,b\xi _{\tau '}\xi\eta ]=0.$$ Hence $(a\xi _{\pi }\; |\; b\xi _{\tau })=0$ by Lemma 19.
        
\vspace{0.8em}

        \item Finally, suppose that $\pi\cap \tau = \emptyset$ but $\pi\dot{\cup}\tau \neq \{ 1,\dots ,n \}$. Let \\ $i\in \{ 1,\dots ,n \}\setminus (\pi\dot{\cup}\tau )$. Consider again the elements $\xi = \frac{1}{2}(\xi _i + \sqrt{-1}\xi _k)$,$\eta = \frac{1}{2}(\xi _i - \sqrt{-1}\xi _k), \; h= 2\xi\eta $. Then the element $b\xi _{\tau }$ is a linear combination of elements $b\xi _{\tau '}\xi,\;  b\xi _{\tau '}\eta,\; [h,b\xi _{\tau '}\xi ]=\pm b\xi _{\tau '}\xi,\; [h,b\xi _{\tau '}\eta ]=\pm b\xi _{\tau '}\eta,\; [h,a\xi _{\pi }] = 0$. Again by Lemma 19, $(a\xi _{\pi }\; |\; b\xi _{\tau })=0$.
    \end{enumerate}

    This completes the proof of the lemma.
\end{proof}




\bibliographystyle{elsarticle-num}

\end{document}